\documentclass{article}%

\usepackage{graphicx}
\usepackage{amsmath,amssymb,amsfonts}%
\usepackage{theorem}%
\usepackage{color}%
\usepackage{hyperref}
\usepackage[font={small, it}]{caption}

\usepackage[all]{xy}

\theoremstyle{change}%
	\newtheorem{definition}{Definition:}[section]%
	\newtheorem{proposition}[definition]{Proposition:}%
	\newtheorem{theorem}[definition]{Theorem:}%
	\newtheorem{lemma}[definition]{Lemma:}%
	\newtheorem{corollary}[definition]{Corollary:}%
	{\theorembodyfont{\rmfamily}\newtheorem{remark}[definition]{Remark:}}%
	{\theorembodyfont{\rmfamily}\newtheorem{example}[definition]{Example:}}%
	
	\newenvironment{proof}
	{{\bf Proof:}}
	{\qquad \hspace*{\fill} $\Box$}%

	\newcommand{\id}{\operatorname{id}}

	\newcommand{\JC}{\mathcal{J}}%
	\newcommand{\CC}{\mathcal{C}}%
	\newcommand{\OC}{\mathcal{O}}%
	\newcommand{\UC}{\mathcal{U}}%
	\newcommand{\AC}{\mathcal{A}}%
	\newcommand{\KC}{\mathcal{K}}%
	\newcommand{\BC}{\mathcal{B}}%
	\newcommand{\N}{\mathbb{N}}%

\begin{document}

	\title{Dynamics on Hyperspaces}
	\author{V\'{\i}ctor Ayala \thanks{%
			Supported by Proyecto Fondecyt $n^{o}$ 1150292, Conicyt, Chile} \ and
		Heriberto Rom\'{a}n-Flores \thanks{%
			Supported by Proyecto Fondecyt $n^{o}$ 1151159, Conicyt, Chile} \\
		Universidad de Tarapac\'a\\
		Instituto de Alta Investigaci\'on\\
		Casilla 7D, Arica, Chile\\
		and\\
		Adriano Da Silva \thanks{%
			Supported by Fapesp grant n%
			${{}^o}$
			2018/10696-6.}\\
		Instituto de Matem\'atica,\\
		Universidade Estadual de Campinas\\
		Cx. Postal 6065, 13.081-970 Campinas-SP, Brasil.\\
	}
	\date{\today }
	\maketitle
	
	\begin{abstract}
	Given a compact metric space $(X, \varrho)$ and a continuous function $f: X\rightarrow X$, we study the
	dynamics of the induced map $\bar{f}$ on the hyperspace of the compact subsets of $X$. We show how the chain recurrent set of $f$ and 
	its components are related with the one of the induced map. The main result of the paper proves that, under mild conditions, the numbers of chain components of  $\bar{f}$ is greater than the ones of $f$. Showing the richness in the dynamics of $\bar{f}$ which cannot be perceived by $f$.
	\end{abstract}
	
	

	\section{Introduction}

Let $(X,\varrho)$ be a compact metric space and $f:X\rightarrow X$ a
continuous function. Our main goal here is to analyze the dynamic behavior of
$f$. However, in real life, the knowledge of the system $f$ is not alway
exact. In fact, each measurement of data carries an uncertainty and it is much
convenient to consider a neighborhood of the state to be measured. In order to
do that, we study the induced function $\bar{f}:\mathcal{K}(X)\rightarrow
\mathcal{K}(X)$ on the hyperspace of nonempty compact subsets of $X$ endowed
with the Hausdorff metric$.$ According to Bauer et al \cite{BAUER}
\textquotedblleft The elements of $\mathcal{K}(X)$ can be viewed as
statistical states, representing imperfect knowledge of the system. The
elements of $X$ are imbedded in $\mathcal{K}(X)$ as the pure
states (see also \cite{ROMAN3}).

In this work we investigate some topological dynamics properties of $f$ and
$\bar{f}$. Precisely, the chain recurrence set of $\bar{f},$ its chain
recurrence components and the relationship with the chain recurrent set of
$f.$

There are many works describing the dynamic behavior of the map $f$. Some of
them analyze the relations between the dynamic of $f$ and those of $\bar{f}$,
\cite{BANKS},\cite{ROMAN1},and \cite{ROMAN3}. In the paper \emph{Chain
	transitivity in hyperspaces}, \cite{LFCGMAAR}, the authors explore which of
these topological properties of $f$ are inherited by $\bar{f}$, and
reciprocally, especially those related to chain transitivity. Our approach is novel and it based on the Conley theory and the
functor $\mathcal{K},$ which works inside of the category of compact spaces
and continuous functions. As a matter of fact, several objects and properties
of $\bar{f}$ can be recovered by this functor through the decomposition of the
recurrent set $\mathcal{C}$ of $f$ \ by its attractors. In fact, if $A$ is an
attractor,\ the dual $A^{\ast}$ of $A$ defined by
\[
A^{\ast}=\left \{  x\in X:\mathcal{O}(x,f)\cap A=\varnothing \right \}
\]
is a repellor and $(A,A^{\ast})$ is called an attractor-repellor pair for $f.$
We prove that $(\mathcal{K}(A),\mathcal{K}(A^{\ast}))$ is an atractor-repellor
pair for the extension $\bar{f}$. And, we do not know if each $\bar{f}%
$-attractor comes from an $f$-attractor.

Through the Conley relationship between the recurrence set $C$ of the original
function $f$ and its attrators given by
\[
\mathcal{C}=\cap \text{ }\left \{  A\cup A^{\ast}:A\text{ is an attractor for
}f\right \}  ,
\]
we show that $\mathcal{K}(\mathcal{C})=$ $\overset{-}{\mathcal{C}}$.
Furthermore, we analyze how the chain recurrent components of $\mathcal{C}$
and $\overset{-}{\mathcal{C}}$ are related. It turns out that under mild
conditions the numbers of chain components of $\overset{-}{\mathcal{C}}$ is
bigger than the ones of $\mathcal{C}$. This fact shows a richness in the
dynamics of $\bar{f}$ which cannot be perceived by the dynamic of $f.$ But,
what are the relations between the chain recurrent components of $\mathcal{C}$
and those of $\overset{-}{\mathcal{C}}?$ Let us denote by $B$ the set of chain
components of $f.$ In order to approach the chain recurrent components of
$\bar{\mathcal{C}}$ we introduce the sets
\[
\mathcal{C}_{\mathcal{J}}:=\left \{  A\in \mathcal{K}(X):\;A\subset \bigcup
_{P\in \mathcal{J}}P\; \mbox{ and }\;A\cap P\neq \emptyset,\;P\in \mathcal{J}%
\right \}  ,
\]
when $\mathcal{J}\in \mathcal{K}(\mathcal{B})$. Then, we prove that each
$\mathcal{C}_{\mathcal{J}}$ is $\bar{f}$-invariant, compact and the disjoint
union of these sets over $\mathcal{K}(\mathcal{B})$ cover $\bar{\mathcal{C}%
},i.e.,$
\[
\bar{\mathcal{C}}\text{ }=\dot{\bigcup_{\mathcal{J}\in \mathcal{K}%
		(\mathcal{B})}}\mathcal{C}_{\mathcal{J}}.
\]
This fact allow to prove our main results.

\begin{theorem}
	The set $\mathcal{C}_{\mathcal{J}}$ is a chain component of $\bar{f}$ if
	$\mathcal{K}(P)$ is chain transitive for any $P\in \mathcal{J}$. In particular,
	the family $\left \{  \mathcal{C}_{\mathcal{J}}:\mathcal{J}\in \mathcal{K}%
	(\mathcal{B}\right \}  $ coincide with the chain components of $\bar{f}$ if
	$\mathcal{K}(P)$ is chain transitive for any $P\in \mathcal{B}$.
\end{theorem}

In particular, if $\mathcal{K}(P)$ is chain transitive for any $P\in
\mathcal{B}$ then
\[
|\mathcal{B}|<\infty \; \; \implies \; \;|\bar{\mathcal{B}}|\text{ }%
=2^{|\mathcal{B}|}-1.
\]

It is important to notice that the chain transitivity of $\mathcal{K}(P)$ does
not follows directly from the chain transitivity of $P$. Actually, in
\cite{LFCGMAAR} the authors give several equivalences for the transitivity of
$\mathcal{K}(P).$

The paper is structured as follows. In Section $2,$ we mention all the
prerequisites needed to state and prove the main results of the paper. Its
include the chain recurrent set of $f$ and its components, the notion of an
attractor-repellor pair, some topological properties of functions on
hyperspaces and the relationship between extension and the canonical
projection $\pi:\mathcal{K}(X)\rightarrow X.$ In Section $3,$ we state our
main results about the chain recurrent components of $\bar{f}.$ Finally, we
give a couple of examples. One of them shows that $\bar{\mathcal{B}}$ is not
numerable depite the fact that $\mathcal{B}$ it is.
		
	\section{Preliminaries}
	
	In this section we state all the prerequisites needed in order to understand the main results of the paper.

	\subsection{The chain recurrent set and its components}
	
	Let $(X, \varrho)$ be a compact metric space and $f:X \rightarrow X$ a continuous function. For any $n> 0$ we denote by $f^n$ the $n$-th iterated of $f$ and $f^0=\id_{X}$. The \emph{ forward orbit } of $x\in X$ is the sequence $f^n(x)$, $n\geq 0$ and is denoted by $\OC(x, f):=\{f^n(x), \;\;n\geq 0\}$. If $A\subset X$ is nonempty, then its \emph{omega-limit set} under $f$ is 
	$$
	\omega(A, f):=\bigcap_k\overline{\bigcup_{n\geq k}f^n(A)}.
	$$
	If $U$ is an open neighborhood of $\omega(A, f)$ then $f^n(A)\subset U$ for $n$ sufficiently large. Actually the omega-limit set is the smallest close set with this property. If $A=\{x\}$ is a singleton, we denote its omega-limit set only by $\omega(x, f)$. A subset $A$ is called \emph{forward invariant} under the map $f$ if $f(A)\subset A$ and {\it invariant} if $f(A)=A$. Note that $\omega(A, f)$ is the largest invariant subset of $f$ contained in $\overline{\OC(A, f)}$. 
	
	Our investigation concerns Conley's \emph{chain recurrent set}, whose definition follows: Given $\varepsilon>0$ an $\epsilon$-\emph{chain} $\xi$ is a finite subset  $\xi=\{x_0, x_1, \ldots, x_n\}$ with $n>0$ such that $\varrho(f(x_{j-1}), x_j))<\varepsilon$ for $0<j<n$. We say that the $\epsilon$-chain starts at $x_0$ and ends at $x_n$, and the integer $n>0$ is its \emph{length}. We denote the set of points that are ends of $\epsilon$-chains beginning at $x$ by $\CC_{\epsilon}(x, f)$, and define 
	$$
	\CC(x, f):=\bigcap_{\epsilon>0}\CC_{\epsilon}(x, f).
	$$ 
	If $0<\delta<\epsilon$ then $\overline{\CC_{\delta}(x, f)}\subset\CC_{\epsilon}(x, f)$ and therefore $\CC(x, f)$ is a closed set. A point $x\in X$ is said to be {\it chain recurrent} if $x\in\CC(x, f)$. The set of the chain recurrent points, denoted by $\CC(f)$, is a nonempty, compact invariant subset of $X$. In $\CC(f)$ one defines the relation $\sim$ by
	$$x\sim y\;\;\mbox{ if }\;\;x\in\CC(y, f)\;\;\mbox{ and }\;\;y\in\CC(x, f).$$ 
	The associated equivalence classes are called the {\it chain components} for $f$. Each of such components is closed and invariant for $f$. 
	
	Let us denote by $\BC_f$ the set of chain components of $\CC(f)$, that is, 
	$$\BC_f:=\{P\in\KC(X), \;\;P \;\mbox{ is a chain component of }\; \CC(f)\}.$$
	The quotient map from $\CC(f)$ to $\BC_f=\CC(f)/\sim$ induces upon the latter the structure of a compact metrizable space. 
	
	We finish this section with a technical lemma which will be needed ahead.
		
	\begin{lemma}
		\label{easy}
		It holds:
		\begin{itemize}
			\item[1.] Let $x_n, y_n\in X$ and assume that $x_n\sim y_n$. If $x_n\rightarrow x$ and $y_n\rightarrow y$ then $x\sim y$;
			\item[2.] 	Let $x, y\in \CC$ be distinct elements. Then, there exists an $\epsilon$-chain between $x$ and $y$ contained in $\CC$ if and only if there exists an $\epsilon$-chain between $y$ and $x$ contained in $\CC$. 
		\end{itemize}
	\end{lemma} 
	
	\begin{proof}
		1. Let $\epsilon>0$ and consider $\delta\in (0, \epsilon/2)$ such that
		$$\varrho(x, y)<\delta\;\;\implies\;\;\varrho(f(x), f(y))<\frac{\epsilon}{2}.$$
		Let $n\in\N$ such that $\varrho(x_n, x)<\delta$ and $\varrho(y_n, y)<\delta$ and let $\xi=\{z_0, \ldots, z_k\}$ to be an $\epsilon/2$-chain between $x_n$ and $y_n$. Since $x_n=z_0$ and $y_n=z_k$ we have that
		$$\varrho(f(x), z_1)\leq\varrho(f(x), f(x_n))+\varrho(f(z_0), z_1)<\frac{\epsilon}{2}+\frac{\epsilon}{2}=\epsilon$$
		$$\hspace{-1,5cm}\mbox{ and }\;\;\;\;\varrho(f(z_{n-1}), y)\leq\varrho(f(z_{n-1}), z_k)+\varrho(y_n, y)<\frac{\epsilon}{2}+\delta<\epsilon$$
		and hence $\xi'=\{x, z_1, \ldots, z_{n-1}, y\}$ is an $\epsilon$-chain between $x$ and $y$ implying $y\in\CC_{\epsilon}(x, f)$. By the arbitrariness of $\epsilon$ we have $y\in\CC(x, f)$. Analogously we show that $x\in\CC(y, f)$ and therefore $x\sim y$ as stated.
		
		2. Let $x, y\in\CC$ and assume that there exists an $\epsilon$-chain between $x$ and $y$ contained in $\CC$. Since any chain component is $f$-invariant there are  distinct chain components $P_1, \ldots, P_{n}$ with $y\in P_1$, $x\in P_n$ and such that $\varrho(P_i, P_{i+1})<\epsilon$ for $i=1, \ldots, n-1$.
		
		For $j=1, \ldots, n-1$ let us consider $x_i, z_i\in P_i$ and $y_i\in P_{i+1}$ be such that $\varrho(x_i, y_i)<\epsilon$ and $f(z_i)=x_i$. By setting $y_0=y$ and $z_n=x$ let $\xi_i$ be an $\epsilon$-chain in $P_i$ between $y_{i-1}$ and $z_i$, for $i=1, \ldots, n$. The inequalities   
		$$\varrho(f(z_i), y_i)=\varrho(x_i, y_i)<\epsilon, \;\;\;i=1, \ldots, n-1,$$
		implies that $\xi=\xi_1\cup\xi_2\cup\ldots\cup\xi_n$ is an $\epsilon$-chain between $y$ and $x$ contained in $\CC$ as stated (see Figure \ref{fig1}).
	\end{proof}

		\begin{figure}[h!]
			\begin{center}
				\includegraphics[scale=1.5]{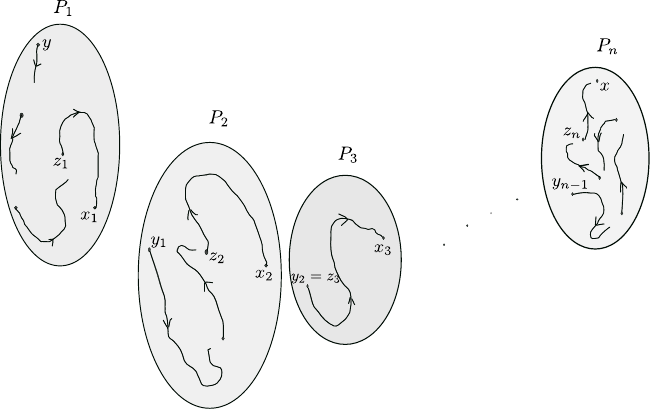}
			\end{center}
			\caption{$\epsilon$-chain between $y$ and $x$.}
			\label{fig1}
		\end{figure}


	\subsection{Attractors-repellor pairs}
	
	There is an important connection between the structure of the set of chain components of $f$ and the set of attractors of $f$. In order to describe this, we first define a closed subset $U\subset X$ to be a {\it trapping region} for $f$ if
	$$\overline{f(U)}\subset U.$$
	For any trapping region $U$ for $f$, the associated {\it attractor} $A$ for $f$, is defined by $\omega(U, f)$, which is nothing more than $\cap_{n\geq 1}f^n(U)$.
	Thus, an attractor $A$ for $f$ is a closed, invariant set which is the limit of the decreasing sequence of iterates $\{f^n(U)\}, \;\;n\geq 1$ for some trapping region $U$. Moreover, the numbers of attractors for $f$ in $X$ is at most countable (see \cite{EAJKMH}).
	
	Let us notice that if $U$ is a trapping region for $f$ with attractor $A$, then any chain recurrent point that is contained in $U$ must actually lie in $A$. 
	
	Suppose now that $U$ is inward for $f$ with attractor $A$. The {\it repellor dual} of $A$ is the set 
	$$A^*:=\{x\in X; \;\mathcal{O}(x, f)\cap U=\emptyset\}.$$
	The pair $(A, A^*)$ is called an {\it attractor-repellor pair}. The relationship between attractor-repellor pairs and $\CC(f)$ comes from the fact that (see \cite{CC}, Chap. II, 6.2.A.)
	\begin{equation}
	\label{chaintheorem}
	\CC(f)=\bigcap\{A\cup A^*; \;\;A \;\mbox{ is an attractor for }\;f\}.
	\end{equation}
	
	\subsection{Functions on hyperspaces}
	
	Let $(X, \varrho)$ be a compact metric space. For any subset $A\subset X$ we denote by  
	$$\KC(A):=\{B\;\mbox{ is a compact subset of }X\;\mbox{ and }\; B\subset A\},$$
	the set of compact subsets of $A$. We also consider the decomposition of $\KC(A)$ in subsets with finite cardinality as follows,
	$$F(A):=\bigcup_{n\geq 1}F_n(A)\;\;\mbox{ with }\;\;F_n(A):=\left\{B\in\KC(A); \;|B|=n\right\}.$$
	It is a standard fact that $F(A)$ is dense in $\KC(A)$ when $A$ is a compact subset of $X$.
	
	It is well known that $(\KC(X), \varrho_H)$ is a metric space, where $\varrho_H$ is the Hausdorff distance given by 
	$$\varrho_H(A, B):=\inf\{\epsilon>0; A\subset N_{\epsilon}(B)\;\mbox{ and }B\subset N_{\epsilon}(A)\},$$
	where $N_{\epsilon}(A)$ stands for the $\epsilon$-neighborhood of $A$. 
	
	The {\it finite topology (or Vietoris topology)} in $\KC(X)$ is the topology generated by the open sets of the form 
	$$\langle U_1, \ldots, U_m\rangle:=\left\{A\in\KC(X);\;\;A\subset\bigcup_{i=1}^mU_i\;\mbox{ and }\; A\cap U_i\neq\emptyset, \;\;\mbox{ for all }i=1, \ldots, n\right\}.$$
	Its a well known fact that the finite topology and the topology induced by the Hausdorff metric are equivalents when $X$ is a compact metric space (see for instance \cite{EM}).	
	
	Let $f:X\rightarrow X$ be a continuous function. The {\it extension} $\bar{f}$ of $f$ from $X$ to $\KC(X)$ is the continuous function 
	$$A\in\KC(X)\mapsto\bar{f}(A):=f(A)\in\KC(X).$$

	The next lemma states several topological properties of hyperspaces that will be needed ahead.	
\begin{lemma}
	\label{lema}
	It holds:

	\begin{itemize}
		\item[1.] Let $\{A_\alpha\}_{\alpha\in\Delta}$ be a family of subsets of $X$. Then
		$$\KC\left(\bigcap_{\alpha\in\Delta}A_{\alpha}\right)=\bigcap_{\alpha\in\Delta}\KC(A_{\alpha}) \;\;\;\mbox{ and }\;\;\;\KC\left(\bigcup_{\alpha\in\Delta}A_{\alpha}\right)=\bigcup_{\alpha\in\Delta}\KC(A_{\alpha}).$$
		
		\item[2.] $\KC(X\setminus U)=\KC(X)\setminus\KC(U)$;\\
		
		\item[3.] $\overline{\KC(U)}=\KC(\overline{U});$\\
	
		\item[4.]$\bar{f}(\KC(U))\subset\KC(f(U))$ for any subset $U\subset X$.
	\end{itemize}

\end{lemma}

	\begin{proof}
		1. $$B\in \KC\left(\bigcap_{\alpha\in\Delta}A_{\alpha}\right)\iff B\subset \bigcap_{\alpha\in\Delta}A_{\alpha}\iff B\subset A_{\alpha}, \;\mbox{ for all }\alpha\in\Delta$$
		$$\iff B\in \KC\left(A_{\alpha}\right),\; \mbox{ for all } \alpha\in\Delta\iff B\in \bigcap_{\alpha\in\Delta}\KC\left(A_{\alpha}\right)$$
		
		$$B\in \KC\left(\bigcup_{\alpha\in\Delta}A_{\alpha}\right)\iff B\subset \bigcup_{\alpha\in\Delta}A_{\alpha}\iff B\subset A_{\alpha}, \;\mbox{ for some }\alpha\in\Delta$$
		$$\iff B\in \KC\left(A_{\alpha}\right),\; \mbox{ for some  }\alpha\in\Delta\iff B\in \bigcup_{\alpha\in\Delta}\KC\left(A_{\alpha}\right)$$
		
		2. 	$$B\in \KC(X\setminus U)\iff B\subset X\setminus U\iff B\not\subset U\iff B\in \KC(X)\setminus \KC(U)$$
		
		3. 	It is straightforward to see that $\KC(\overline{U})$ is a closed subset of $\KC(X)$ and hence $\overline{\KC(U)}\subset \KC(\overline{U})$. Reciprocally, let $B\in \KC(X)\setminus\overline{\KC(U)}$. There exists a neighborhood $V$ of $B$ such that $\langle V\rangle\cap \KC(U)=\emptyset$. Therefore,
		$$\langle V\rangle\cap \KC(U)=\KC(V)\cap \KC(U)=\KC(V\cap U)$$
		implying that $V\cap U=\emptyset$ and consequently that $B\in\KC(X)\setminus\KC(\overline{U})$.
		
		4. 	$$B\in \bar{f}(\KC(U))\implies B=\bar{f}(A),\;A\in \KC(U)$$
		$$\implies B=f(A), \;A \mbox{ compact subset of }U\implies B\in\KC(f(U)).$$
	\end{proof}
	
	In the sequel we define a projection from the space of subsets $\KC(X)$ to $X$ that still preserve several topological properties. For any subset $\UC\subset\KC(X)$ we can define its $X$-{\it projection} by 
		$$\pi(\UC):=\{x\in X; \exists B\in\UC \;\mbox{ with }\;x\in B\}\subset X.$$
		The next proposition gives us the main topological relation between $\UC$ and its $X$-projection.
		
		\begin{proposition}
			\label{prop}
			It holds:
			\begin{itemize}
			\item[1.] If $\UC$ is open in $\KC(X)$ then $\pi(\UC)$ is an open $X$;\\
			\item[2.] $\overline{\pi(\UC)}=\pi(\overline{\UC})$;\\
			\item[3.]  If $\UC_1\subset\UC_2$ then $\pi(\UC_1)\subset \pi(\UC_2)$;\\
		    \item[4.]  $f(\pi(\UC))=\pi(\bar{f}(\UC));$\\
		    \item[5.] If $A\in\KC(X)$ then $\pi(\KC(A))=A$.\\
			\end{itemize}
		 	\end{proposition}

			\begin{proof}
				1. Let $x\in \pi(\UC)$ and consider $B\in\UC$ such that $x\in B$. Since $\UC$ is open, there exists a compact neighborhood $V$ of $B$ such that $\langle V\rangle\subset\UC$. In particular, $V\in\UC$ implying that $x\in B\subset V\subset\pi(\UC)$ and showing that $x$ is an interior point of $\pi(\UC)$.
				
				2.	{\bf $\pi(\overline{\UC})$ is closed:} Let $x_n\in\pi(\overline{\UC})$ such that $x_n\rightarrow x$ and consider $B_n\in \overline{\UC}$ with $x_n\in B_n$. By the compacity of $\KC(X)$, there exists $B_{n_k}$, $k\in\N$ such that $B_{n_k}\rightarrow B$. Since $\overline{\UC}$ is closed, we must have $B\in\overline{\UC}$ which implies $x\in B$ and therefore $x\in\pi(\overline{\UC})$.  \\
				
				{\bf $\pi(\UC)$ is dense in $\pi(\overline{\UC})$:} Let $x\in\pi(\overline{\UC})$ and $B\in\overline{\UC}$ with $x\in B$. For any $\varepsilon>0$ there exists $B'\in\UC$ with $\varrho^H(B', B)<\varepsilon$. Therefore, there exists $x'\in B$ such that $\varrho(x', x)<\varepsilon$. In particular, it holds $x'\in\pi(\UC)$ and consequently $\pi(\UC)$ is dense in $\pi(\overline{\UC})$.
				
				3. It holds that $x\in\pi(\UC_1)\iff \exists B\subset\UC_1; \;x\in B$. Since $\UC_1\subset\UC_2$ we get that $B\in\UC_2$ implying that $x\in\pi(\UC_2)$.
				
				4. $$x\in f(\pi(\UC))\iff x=f(y), \;y\in \pi(\UC)\iff \exists B\in\UC\;\mbox{ with }\;x=f(y)\;\mbox{ for some }\; y\in B$$
				$$\iff x\in \bar{f}(B) \mbox{ and }B\in\UC\iff x\in A \mbox{ and }A\in\bar{f}(\UC)\iff x\in \pi(\bar{f}(\UC))$$
				
				5. Since $A\in \KC(A)$ it follows that $A\subset \pi(\KC(A))$. Reciprocally, for any $x\in\pi(\KC(A))$ there is $B\in\KC(A)$ with $x\in B$. However, $B\in\KC(A)$ if and only if $B\subset A$ implying that $x\in A$ and concluding the proof. 
				\end{proof}
	
	\bigskip
	
	Now, we are able to show the relationship between the attractors and repellers of $f$ with those of $\bar{f}$.

\begin{proposition}
	If $A\subset X$ is an attractor, then $\mathcal{K}(A)$ is an attractor. Reciprocally, if $\AC\subset\KC(X)$ is an attractor, then $\pi(\AC)$ is an attractor.
\end{proposition}

\begin{proof}
	Let $U$ be a trapping region for $f$ with attractor $A$. From Lemma \ref{lema} $\KC(U)$ is closed and
	$$\overline{\bar{f}(\UC)}=\overline{\bar{f}(\KC(U))}\subset \overline{\KC\left(f(U)\right)}=\KC\left(\overline{f(U)}\right)\subset\KC(U)=\UC$$
	which implies that $\KC(U)$ is a trapping region for $\bar{f}$. Moreover, since $A$ is $f$-invariant, we have that $\KC(A)$ is $\bar{f}$-invariant and so $\KC(A)\subset\bigcap_{n\in\N}\bar{f}^n(\UC)$. On the other hand, let $B\in \bigcap_{n\in\N}\bar{f}^n(\UC)$. For any $n\in\N$ there exists $B_n\subset U$ such that $B=\bar{f}^n(B_n)$. Therefore, for all $n\in\N$ we get
	$$B=f^n(B_n)\subset f^n(U) \implies B\subset \bigcap_{n\in\N}f^n(U)=A\implies B\in\KC(A)$$
	showing that $\KC(A)=\bigcap_{n\in\N}\bar{f}^n(\UC)$ and consequently $\KC(A)$ is an attractor.
	
	Reciprocally, let $\UC$ be a trapping region for $\bar{f}$ with attractor $\AC$ and consider $U=\pi(\UC)$. By Proposition \ref{prop} we obtain that $U$ is closed and
	$$\overline{f(U)}=\overline{f(\pi(\UC)})=\overline{\pi(\bar{f}(\UC))}=\pi\left(\overline{\bar{f}(\UC)}\right)\subset \pi(\UC)=U$$
	showing that $U$ is a trapping region for $f$. Moreover, since $\AC$ is $\bar{f}$-invariant, we have that $\pi(\AC)$ is $f$-invariant and so $\pi(\AC)\subset\bigcap_{n\in\N}f^n(U)$. On the other hand, if $x\in\bigcap_{n\in\N}f^n(U)$, for each $n\in\N$ there exists $B_n\in \UC$ such that $x\in f^n(B_n)$. Furthermore, for some subsequence $(B_{n_k})$ we have that $\bar{f}^{n_k}(B_{n_k})\rightarrow B$ for some $B\in\KC(X)$. In particular, $B\in\omega(\UC, \bar{f})=\AC$ which implies that $B\subset \pi(\AC)$. Since $x\in f^{n_k}(B_{n_k})$ we get $x\in B\subset \pi(\AC)$ implying that $\bigcap_{n\in\N}f^n(U)=\pi(\AC)$ and ending the proof.	
\end{proof}

\bigskip

Concerning the dual repellors we have the following result.

	\begin{proposition}
		Let $A\subset X$ and consider its associated repeller $A^*$. Then 
		$$\mathcal{K}(A)^*=\mathcal{K}(A^*).$$
	\end{proposition}
	
	\begin{proof}
	   As before, if $U$ is a trapping region for $f$ with attractor $A$ then $\UC=\KC(U)$ is a trapping region for $\bar{f}$ with attractor $\KC(A)$. Therefore,  
	   $$\OC(B, \bar{f})\cap \UC=\emptyset\Leftrightarrow\forall n\geq 0, \;\bar{f}^n(B)\in\KC(X)\setminus\UC\Leftrightarrow \forall n\geq 0, \;f^n(B)\subset X\setminus U\Leftrightarrow $$
	   $$\forall n\geq 0, \;f^{n}(B)\cap U=\emptyset\Leftrightarrow \OC(x, f)\cap U=\emptyset, \;\mbox{ for all } \;x\in B.$$
		
	We obtain, 
	$$B\in\KC(A)^*\Leftrightarrow \OC(B, \bar{f})\cap \UC=\emptyset\Leftrightarrow  \OC(x, f)\cap U=\emptyset, \mbox{ for all } x\in B$$
	$$\Leftrightarrow B\subset A^*\Leftrightarrow B\in\KC(A^*)$$
		concluding the proof.		
	\end{proof}

	\section{Chain recurrence on Hyperspaces}
	
	In this section we prove our main results relating the chain recurrent sets of $f$ and its extension $\bar{f}$. We also show that, although there is a close relation between the chain components of the recurrent sets of $f$ and $\bar{f}$, however their grows exponentially.

	\subsection{The chain recurrent set}
	
	From here on we denote only by $\CC$ and $\bar{\CC}$ the chain recurrent sets of $f$ and $\bar{f}$, respectively. The next theorem relates $\CC$ and $\bar{\CC}$.
 	
 	\begin{theorem}
 		\label{main}
 		It holds that $\mathcal{K}(\CC)=\bar{\CC}.$ 
 	\end{theorem}
 	
 	\begin{proof}  
 		We start to prove that $\mathcal{K}(\CC)\subset\bar{\CC}$. Since $F(\CC)$ is dense in $\mathcal{K}(\CC)$ and $\bar{\CC}$ is closed, it is enough to show that $F(\CC)\subset \bar{\CC}$. Moreover, the fact that $F(\CC)=\bigcup_{n\in\N}F_n(\CC)$ implies that is enough to show that $F_n(\CC)\subset \bar{\CC}$ for each $n\in\N$, which we prove by induction on $n\in\N$.
 		
 		It is easy to see that $F_1(\CC)$ is chain recurrent. Let us then assume that $F_n(\CC)$ is chain recurrent and consider $A=\{x_1, \ldots, n_{n+1}\}\in F_{n+1}(\CC)$. By the induction hypothesis, for a given $\epsilon>0$ there exists a chain $\xi_1$ in $F_n(\CC)$, $B'_0, \ldots, B'_{k_1}\in F_n(\CC_f)$ with $B'_0=B'_n=\{x_1, \ldots, x_n\}$ and $\varrho_H(\bar{f}(B_i), B_{i+1})<\epsilon$ for $i=0, \ldots, k_1$ and a chain $\xi_2$ in $F_1(\CC)$, $\{a_0\}, \ldots, \{a_{k_2}\}\subset X$ with $a_0=a_{k_2}=x_{n+1}$ and $\varrho(f(a_i), a_{i+1})<\epsilon$ for $i=0, \ldots, k_2$. By concatenating $k_2$-times  $\xi_1$ and $k_1$-times $\xi_2$ we can assume that $k=k_1=k_2$. Thus, $A_i=B_i\cup \{a_i\}$, $i=0, \ldots, n+1$ is an $\epsilon$-chain from $A$ to itself in $F_{n+1}(\CC_f)$. Since $\epsilon$ was arbitrary we get that $F_{n+1}(\CC)\subset\bar{\CC}$ implying that $\KC(\CC)\subset \bar{\CC}$.		
 		
 		Reciprocally, we need to show that $\mathcal{K}(\CC)\supset\CC_{\bar{f}}$. Using the relation (\ref{chaintheorem}) we get 
 		$$\KC(C)=\KC\left(\bigcap\{A\cup A^*, \;A \mbox{ is an attractor for }f\}\right)$$
 		$$\bigcap\{\KC(A\cup A^*), \;A \mbox{ is an attractor for }f\}=$$
 		$$\bigcap\{\KC(A)\cup\KC(A^*), \;A \mbox{ is an attractor for }f\}=$$
 		$$\bigcap\{\KC(A)\cup\KC(A)^*, \;A \mbox{ is an attractor for }f\}\supset$$
 		$$\bigcap\{\AC\cup\AC^*, \;\AC \mbox{ is an attractor for }\bar{f}\}=\bar{\CC}$$
 		which ends the proof. 		
 	\end{proof}
 	
 	\bigskip
 	
 	As a direct corollary we have the following:
 	
 	\begin{corollary}
 		Let $X$ be a compact metric space and $f:X\rightarrow X$ a continuous function. Then, $f$ is chain recurrent if and only if $\bar{f}$ is chain recurrent.
 	\end{corollary}
 	
 	\begin{remark}
 		The previous corollary shows that the chain recurrence of $f$ and $\bar{f}$ are equivalent. However, the same does not holds for the chain transitivity as stated in \cite{LFCGMAAR}.
 	\end{remark}

 	\subsection{Chain recurrent components}
 	
	In this section we analyze the relation between the chain recurrent components of $\CC$ and of $\bar{\CC}$. We show that under mild conditions the numbers of chain components of $\bar{\CC}$ is much greater than the one of $\CC$ which implies a richness in the dynamics of $\bar{f}$ which cannot be perceived by $f$.
 	
 	Let us denote by $\BC$ the chain components of $f$ and by $\KC(\BC)$ the set of its compact subsets. For any $\JC\in\KC(\BC)$ we define the sets  
 	$$\CC_{\JC}:=\left\{A\in\mathcal{K}(X); \;A\subset\bigcup_{P\in \JC} P \;\mbox{ and }\; A\cap P\neq\emptyset, \;P\in \JC\right\}.$$
 	 	
 	Our aim is to relate the sets $\CC_{\JC}$ with the chain recurrent components of $\bar{\CC}$. The next proposition goes in this direction.
 	 	
 	\begin{proposition}
 		The sets $\CC_{\JC}$ are $\bar{f}$-invariant, compact and satisfy 
 		$$\bar{\CC}=\dot{\bigcup_{\JC\in\KC(\BC)}}\CC_{\JC}.$$
 	\end{proposition}
 	
 	\begin{proof}
 	The $f$-invariance of $P\in\BC$ implies that  
    $$f(A)\subset f\left(\bigcup_{P\in\JC}P\right)=\bigcup_{P\in\JC}f(P)\subset \bigcup_{P\in\JC}P\;\;\mbox{ and }\;\;f(A)\cap P\supset f(A\cap P)\neq\emptyset$$
	and hence $\bar{f}\left(\CC_{\JC}\right)\subset\CC_{\JC}$. The disjointness follows from the fact that if $A\in\CC_{\JC_1}\cap\CC_{J_2}$ then  
	$$A\cap P_1\subset \bigcup_{P_2\in\JC_2}P_2\;\;\implies\;\;P_1\cap P_2 \;\;\mbox{ for some }\;\;P_2\in\JC_2\implies \JC_1\subset J_2$$
	and analogously $\JC_2\subset\JC_1$ implying that $\CC_{\JC_1}=\CC_{\JC_2}$.
	
	By Theorem \ref{main}, if $A\in\bar{\CC}$ then $A\subset\CC=\bigcup_{P\in\BC}P$ implying that 
	$$A=\bigcup_{P\in\JC}A\cap P, \;\;\mbox{ where }\;\;\JC:=\{P\in\BC; \;\;A\cap P\neq\emptyset\}.$$
    Moreover, if $P_n\in\JC$ is such that $P_n\rightarrow P$ in $\BC$, the fact that $P_n\cap B\neq\emptyset$ implies the existence of $x_n\in P_n\cap B$. Since $A$ is compact we have that $x_{n_k}\rightarrow x\in A$ and hence $x\in P$ implying $P\cap A\neq\emptyset$ and consequently $P\in\JC$. In particular, $\JC\in\KC(\BC)$ and  
    $$\bar{\CC}=\dot{\bigcup_{\JC\in\KC(\BC)}}\CC_{\JC}.$$
    	
    In order to show the compactness of $\CC_{\JC}$ it is enough to show that it is closed. If $B_n\in\CC_{\JC}$ is such that $B_n\rightarrow B$. We have  
    \begin{itemize}
    	\item[1.] $B\cap P\neq\emptyset$ for any $P\in\JC$;
    	
    	In fact, there exists $x_n\in B_n\cap P$ for any $P\in\JC$ and by compactness $x_{n_k}\rightarrow x\in B\cap P$ implying the assertion.
    	
    	\item[2.] $B\subset \overline{\bigcup_{P\in\JC}P}$;
    	
    	In fact, by the definition of the Hausdorff metric we have that 
    	$$B\subset N_{\epsilon}\left(\bigcup_{P\in\JC}P\right), \;\;\mbox{ for any }\;\;\epsilon>0.$$
    	Since $\bigcap_{\epsilon>0}N_{\epsilon}\left(\bigcup_{P\in\JC}P\right)=\overline{\bigcup_{P\in\JC}P}$ we have the assertion.
    \end{itemize} 
    
    By assertions 1. and 2. above, the closedness of $\CC_{\JC}$ is equivalent to the closedness of $\bigcup_{P\in\JC}P$. 
    
    Let then $x_n\in\bigcup_{P\in\JC}P$ and assume that $x_n\rightarrow x$ in $\CC$. For any $n\in\N$ let $P_n\in\JC$ with $x_n\in P_n$. Since $x_n\rightarrow x$ in $\CC$ we have that $P_n\rightarrow P$ in $\BC$, where $x\in P$. Being that $\JC\in\KC(\BC)$ we must have $P\in\JC$ and consequently $x\in P\subset\bigcup_{P\in\JC}P$ showing that $\bigcup_{P\in\JC}P$ is closed and concluding the proof.   
    \end{proof}
    
    \begin{remark}
    	The above proof shows, that 
    	$$\bigcup_{P\in\JC}P\in\CC_{\JC}, \;\;\mbox{ for any }\;\;\JC\in\KC(\BC)\;\;\implies\;\;\CC_{\JC}\neq\emptyset.$$
    	In particular, each $\CC_{\JC}$ admits a fixed point of $\bar{f}$.
    \end{remark}

 	\bigskip
 	
  The next example shows that such sets can be much greater than one expects.
 	
 	\begin{example}
 		Let us consider $X=\{a, b, c\}$ endowed with the zero-one metric and define $f:X\rightarrow X$ by the relations
 		$$f(a)=b, \;\;\;f(b)=a\;\;\mbox{ and }\;\;f(c)=a.$$
 		For such metric, there are no chains from $a$ or $b$ to $c$ for any $\epsilon<1$ and therefore, $\CC=\{a, b\}$. On the other hand, the Hausdorf metric induced  in $\KC(X)$ is also the zero-one metric and therefore $\bar{\CC}=\{\{a\}, \{b\}, \{a, b\}\}= \KC(\CC)$. 
 		
 		Since $\BC=\{\CC\}$ we have that $\CC_{\BC}=\bar{\CC}$. On the other hand, the chain recurrent components of $\bar{\CC}$ are given by 
 		$$\{\{a\}, \{b\}\}\;\;\;\mbox{ and }\;\;\;\{\{a, b\}\}$$
 		showing that the sets $\CC_{\JC}$ do not need to be chain transitive. 		
 	\end{example}
 	 	
 	In fact, as in the above example, the sets $\CC_{\JC}$ are an ``outer" approximation for the chain components of $\bar{f}$.

 	\begin{proposition}
 		Let $A_i\in\CC_{\JC_i}$, $\JC_i\in\KC(\BC)$ for $i=1, 2$. Then 
 		$$A_1\sim A_2\;\;\;\implies\;\;\;\CC_{\JC_1}=\CC_{\JC_2}.$$
 		In particular, $\CC_{\JC}$ is a disjoint union of chain components of $\bar{f}$.
 	\end{proposition}

 	\begin{proof}
 		Since $A_1, A_2\in\bar{\CC}$ and $A_1\sim A_2$ we have that $A_1$ and $A_2$ are in the same chain component of $\bar{f}$. Therefore, for any $\epsilon>0$ there exists a chain between $A_2$ and $A_1$ contained in $\bar{\CC}$\footnote{Here we used the fact that two points in the same chain component can be joined by a chain contained in the chain component.}. Let then $P\in\JC_1$ and consider $x\in A_1\cap P$ and a $1/n$-chain between $A_2$ and $A_1$. There exists $x_n\in P_n\in\JC_2$ and a $1/n$-chain component betwenn $x_n$ and $x$ contained in $\CC$. By Lemma \ref{easy} there exists an $1/n$-chain between $x$ and $x_n$ contained in $\CC$. Being that $\CC$ is compact, we can assume w.l.o.g. that $x_n\rightarrow x'$ and hence $x'\sim x$ implying that $x'\in P$. Therefore, $P_n\rightarrow P$ in $\BC$ and since $\JC_2$ is closed we must have $P\in\JC_2$. Since $P\in\JC_1$ is arbitrary we get $\JC_1\subset\JC_2$. Analogously we have $\JC_2\subset\JC_1$ and hence $\CC_{\JC_1}=\CC_{\JC_2}$ as stated.  
 	\end{proof}
 	
 		\bigskip
 		
 		The next step is to analyze when the sets $\CC_{\JC}$ are chain recurrent. If $\JC=\{P\}$ is an unitary set, we have that $\CC_{\JC}=\KC(P)$ and by invariance  $\overline{f|_{P}}=\bar{f}|_{\KC(P)}$. For the unitary case, there are several equivalences for the chain transitivity of the induced map (see Theorem 3 of \cite{LFCGMAAR}). Our aim is to show that, in fact, the knowledge of the chain components of $\bar{f}$ are well-known if one have transitivity of the maps restricted to $\KC(P)$ for any chain component $P\in\BC$.
 		
 		Let then $P\in\BC$ be a chain component and assume that $\KC(P)$ is chain transitive. In this case, for any $A, B\in\KC(P)$ there exists $n\in\N$ such that for any $m\geq n$ there is a chain $\xi$ connecting $A$ and $B$ with exactly lenght $m$.
 	
 	In fact, since $P$ is $f$-invariant we have that, as an element of $\KC(P)$, $P$ is a fixed point of $\bar{f}$. For any given $A, B\in\KC(P)$ there exists chains $\xi_1, \xi_2$ connecting $A$ to $P$ and $P$ to $B$ with lenght $n_1, n_2$, respectively. If $n=n_1+n_2$, for any $m\geq n$ we can ``insert" $m-n$ points in the chain $\xi_1$ equals to $P$ and concatenate with the chain $\xi_2$ in order to obtain a chain from $A$ to $B$ with length $m$ as stated.
 	 	
 	Now we are able to show our main result relating the chain components of $f$ and $\bar{f}$. 
 	
 	\begin{theorem}
 		The set $\CC_{\JC}$ is a chain component of $\bar{f}$ if $\KC(P)$ is chain transitive for any $P\in\JC$. In particular, $\CC_{\JC}$ are the chain components of $\bar{f}$ if $\KC(P)$ is chain transitive for any $P\in\BC$.
 	\end{theorem}
 	
 	\begin{proof}
 		Let us assume first that $\JC=\{P_1, \ldots, P_n\}$ and that $\KC(P_i)$ is chain transitive for $i=1, \ldots, n$. If $A, B\in\CC_{\JC}$ we have that 
 		$$A=\dot{\cup}A_i \;\;\mbox{ and }\;\;B=\dot{\cup}B_i\;\;\mbox{ where }\;\;A_i=A\cap P_i \;\;\mbox{ and }\;\;B_i=B\cap P_i, \;\;i=1, \ldots, n.$$
 		By the previous discussion and the assumption on the chain transitiveness of $\KC(P_i)$ there exist $\epsilon$-chains $\xi_i$ from $A_i$ to $B_i$ with same length $n\in\N$ for $i=1, \ldots, n$. The chain $\xi$ given by $\xi_1\cup\cdots\cup\xi_n$ is then an $\epsilon$-chain from $A$ to $B$ showing, by arbitrariness, that $\CC_{\JC}$ is chain transitive.
 		
 		Let us consider now the case $|\JC|=\infty$ and assume that $\KC(P)$ is chain transitive for any $P\in\JC$. For any given $\delta>0$ and any $A\in\CC_{\JC}$ we have that 
 		$$A\subset N_{\delta}\left(\bigcup_{P\in\JC}A\cap P\right)=\bigcup_{P\in\JC}N_{\delta}(A\cap P).$$
 		By compactness, there exists $\JC_{A}=\{P_1, \ldots, P_n\}\subset\BC$ such that 
 		\begin{equation}
 		\label{eq}
 		A\subset \bigcup_{i=1}^n N_{\delta}(A\cap P_i)=N_{\delta}\left(\bigcup_{i=1}^nA\cap P_i\right).
 		\end{equation}
 		
 		Let then $A, B\in \CC_{\JC}$ and consider as above $\JC_A, \JC_B\subset\BC$ satisfying (\ref{eq}). By setting $\JC_{A, B}=\JC_A\cup\JC_B$ we have that  
 		$$A\subset N_{\delta}\left(\bigcup_{P\in\JC_A}A\cap P\right)\subset N_{\delta}\left(\bigcup_{P\in\JC_{A, B}}A\cap P\right)\;\;\mbox{ and }\;\;B\subset N_{\delta}\left(\bigcup_{P\in\JC_B}B\cap P\right)\subset N_{\delta}\left(\bigcup_{P\in\JC_{A, B}}B\cap P\right).$$
 		If 
 		$$A_{\delta}=\bigcup_{P\in\JC_{A, B}}A\cap P \;\;\mbox{ and }\;\;B_{\delta}=\bigcup_{P\in\JC_{A, B}}B\cap P$$
		we have that $A_{\delta}, B_{\delta}\in\CC_{\JC_{A, B}}$ and, since $A_{\delta}\subset A$ and $B_{\delta}\subset B$, the above implies $\varrho_H(A, A_{\delta})<\delta$ and $\varrho_H(B, B_{\delta})<\delta$.
		Moreover, the hypothesis that $\KC(P)$ is chain transitive for any $P\in\JC$ implies by the finite case that $\CC_{\JC_{A, B}}$ is chain transitive and hence $A_{\delta}\sim B_{\delta}$. 
		
		By the arbitrariness of $\delta>0$ there exists $A_n\rightarrow A$ and $B_n\rightarrow B$ with $A_n\sim B_n$ which by Proposition \ref{easy} implies $A\sim B$ and hence $\CC_{\JC}$ is chain transitive.  
	\end{proof}

A direct corollary of the previous results is the following.

\begin{corollary}
	If $\KC(P)$ is chain transitive for any $P\in\BC$ then 
	$$|\BC|<\infty\;\;\implies\;\;|\bar{\BC}|=2^{|\BC|}-1.$$
\end{corollary}

\begin{proof}
	In fact, under the assumption that $\KC(P)$ is chain transitive for any $P\in\BC$ we have that the chain components of $\bar{f}$ are parametrized by the nonempty compact subsets of $\BC$ and hence $|\bar{\BC}|=|\KC(\BC)|=2^{|\BC|}-1$.
\end{proof}

\begin{remark}
	It is important to notice that the chain transitivity of $\KC(P)$ it is not direct from the chain transitivity of $P$. In fact, in \cite{LFCGMAAR} the authors give several equivalence for the transitivity of $\KC(P)$ 
\end{remark}

	Next we give an example where $\BC$ is enumerable and $\bar{\BC}$ is nonenumerable.

	\begin{example}
		Let $X=[0, 1]$ and consider
		$$f:X\rightarrow X, \;\;\mbox{ given by }\;\;f(x):=\left\{\begin{array}{cc}
		x\left|\sin\left(\frac{\pi}{x}\right)\right|, & \;\mbox{ if }x>0\\
		0 & \;\mbox{ if }x=0. 
		\end{array}\right.$$ 
		
		It is not hard to see that 
		$$A=\{0\}\cup\left\{\frac{1}{2k+1}, \;k\geq 0\right\}$$
		is the set of fixed points of $f$ (see Figure \ref{fig2})  and hence $A\subset\CC(f)$. On the other hand, a simple calculation shows that $f^n(x)\rightarrow 0$ for any $x\notin A$ and hence   
		$$A_k:=\left[0, \frac{1}{2k+1}\right]\;\;\mbox{ and }\;\;A_k^*=\left\{\frac{1}{2m+1}, \;\;m\leq k\right\}$$
		is an attractor-repellor pair of $f$. Therefore,
		$$\CC(f)\subset\bigcap_k A_k\cup A_k^*\subset A\;\;\implies\;\;\CC(f)=A,$$
		showing that $f$ has an enumerable number of chain components.
		
		In particular, $\BC=\{\{x\}, \;\;x\in\CC\}$ and since $\KC(\{x\})=\{x\}$ we have that $\KC(\{x\})$ is chain transitive for any $x\in\CC$. By the previous results we have that the chain components of the induced map $\bar{f}$ are of the form $\CC_{\JC}$ with $\JC\in\KC(\BC)\}$. However, 
		$$\BC\sim\CC=\{0\}\cup\left\{\frac{1}{2k+1}, \;k\geq 0\right\}\implies \KC(\BC)=\{\JC\subset A; \;\;|\JC|<\infty\}\cup A$$
		and hence $\bar{f}$ has an unenumerable numbers of chain components.
	\end{example}

	\begin{figure}[h!]
		\begin{center}
			\includegraphics[scale=0.75]{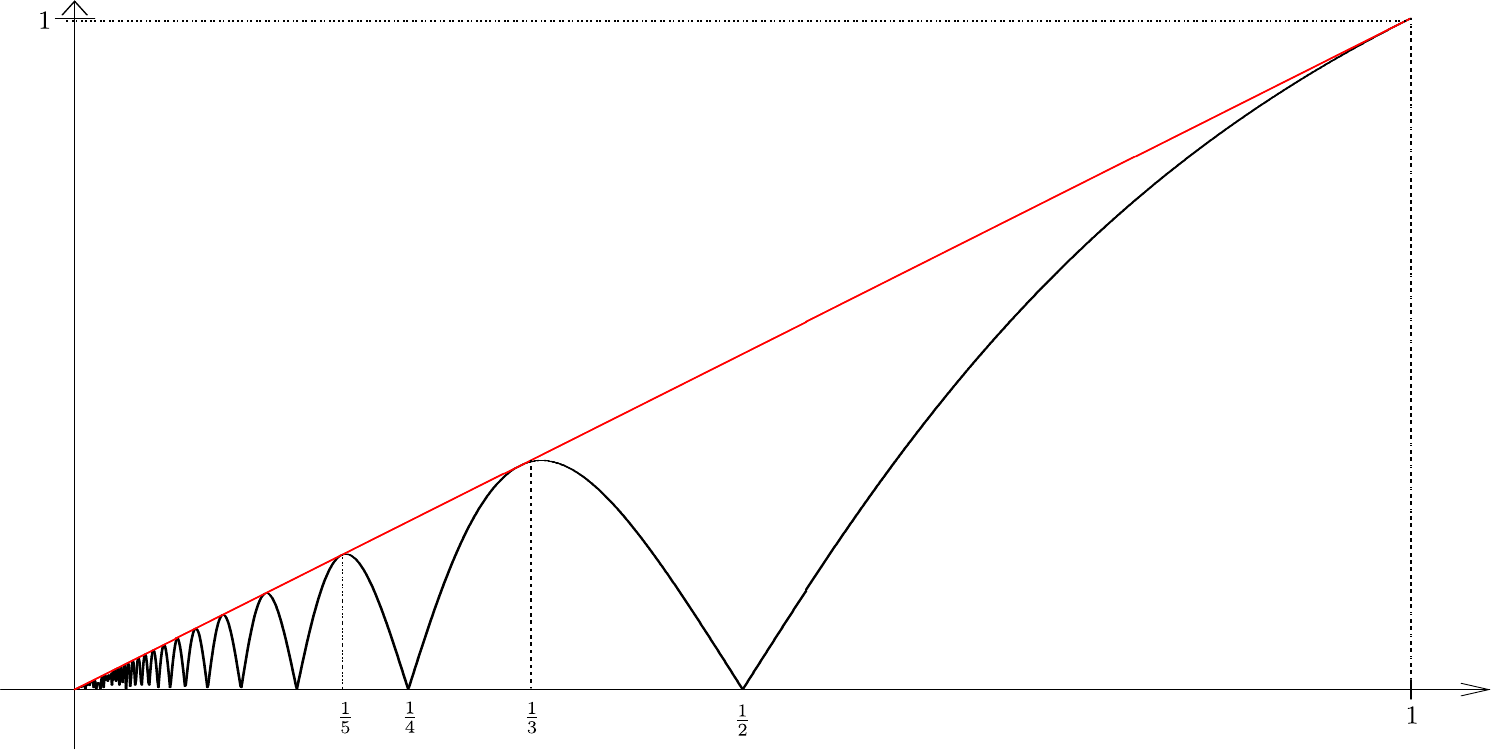}
		\end{center}
		\caption{Graphic of $f$.}
		\label{fig2}
	\end{figure}

	\begin{remark}
		Normally, the complexity of the total space $\mathcal{K}(X)$ is greater than
		the complexity of the base space $X.$ For instance, respect to the
		topological entropy we know that 
		\[
		top_{e}(\mathcal{K}(X),\bar{f})\geq top_{e}(X,f),
		\]
		(see \cite{canovas}). However, in \cite{ROMAN1,ROMAN3} the author consider
		the extention of $f$ to $\bar{f}$ over the class of compact an convex
		subsets of a real interval $X$. In this case, the complexity of $\bar{f}$
		decreases. We intend to continue to search on this topic by considering new
		extentions over distinguished subspaces of $\mathcal{K}(X).$\newline
	\end{remark}

\end{document}